\documentclass[12pt,oneside]{amsart}
\usepackage{graphicx, epsfig}

\newfont{\sheaf}{eusm10 scaled\magstep1}

\newcommand{\ra}{\ensuremath{\rightarrow}}

\def\eea{\end{eqnarray*}}
\def\bea{\begin{eqnarray*}}

\def\de{{\delta}}

\newcommand{\Proof}{{\it Proof. }}
\newcommand{\QED}{{\hfill $Q.E.D.$}}

\newtheorem{teo}{Theorem}[section]
\newtheorem{df}[teo]{Definition}
\newtheorem{lem}[teo]{Lemma}
\newtheorem{cor}[teo]{Corollary}
\newtheorem*{cor*}{Corollary}

\newtheorem{oss}[teo]{Remark}
\newtheorem{rem}[teo]{Remark}
\newtheorem{prop}[teo]{Proposition}

\newcommand{\bbA}{\ensuremath{\mathbb{A}}}
\newcommand{\C}{\ensuremath{\mathbb{C}}}
\newcommand{\R}{\ensuremath{\mathbb{R}}}
\newcommand{\Z}{\ensuremath{\mathbb{Z}}}

\newcommand{\F}{\ensuremath{\mathbb{F}}}

\newcommand{\HH}{\ensuremath{\mathbb{H}}}

\newcommand{\PP}{\ensuremath{\mathbb{P}}}

\newcommand{\SSS}{\ensuremath{\mathcal{S}}}
\newcommand{\PPP}{\ensuremath{\mathcal{P}}}
\newcommand{\KKK}{\ensuremath{\mathcal{K}}}

\begin{document}

\title[Real singular Del Pezzo surfaces]{Real singular Del Pezzo surfaces and
3-folds fibred by  rational curves, II}

\author{Fabrizio Catanese\\
    Universit\"at Bayreuth \\ Frederic Mangolte\\ Universit\'e de Savoie }

\thanks{ The research of the  authors was  supported by the
D.F.G.-FORSCHERGRUPPE 790
"Classification of algebraic surfaces and compact
   complex manifolds" and the ANR grant "JCLAMA" of Agence Nationale de
la Recherche. We profited of the hospitality of the Centro `De Giorgi' of the
Scuola Normale di Pisa in order to meet and prepare the final version
of the article.}

\date{November 5, 2008}

\maketitle

\begin{quote}\small
\end{quote}

\begin{abstract} Let $W \to X$ be a real smooth projective 3-fold
fibred by rational
curves. J.~Koll\'ar proved that, if $W(\R)$ is orientable, then a
connected component
$N$ of $W(\R)$ is essentially  either a Seifert fibred  manifold or a
connected sum of
lens spaces.

Our Main Theorem,  answering in the affirmative three questions  of
Koll\'ar, gives sharp estimates on the number and the multiplicities of the
Seifert fibres and on the number and the torsions of the lens spaces
when $X$ is a
geometrically rational surface.

When $N$ is Seifert fibred
   over a base orbifold $F$, our result  generalizes
 Comessatti's theorem on smooth real rational surfaces: $F$ cannot be
simultaneously orientable and of hyperbolic type.

We show as a surprise  that,  unlike in Comessatti's theorem,
 there are examples where $F$
is non orientable, of hyperbolic type, and $X$ is minimal. 

The technique we use is to construct Seifert fibrations as projectivized
tangent bundles of Du Val surfaces.

\end{abstract}

\medskip
\renewcommand{\abstractname}{Titre et r\'esum\'e en fran\c{c}ais}
\begin{abstract} 

\textbf{Surfaces de Del Pezzo singuli\`eres r\'eelles et vari\'et\'es de dimension 3 munies d'une fibration en courbes rationnelles}

Soit $W \to X$ une vari\'et\'e projective r\'eelle non singuli\`ere munie d'une fibration en courbes rationnelles.  J.~Koll\'ar a montr\'e que si $W(\R)$ est orientable, alors une composante connexe $N$ de $W(\R)$ est essentiellement une vari\'et\'e de Seifert ou une somme connexe d'espaces lenticulaires.

Notre th\'eor\`eme principal, r\'epondant affirmativement \`a trois questions de Koll\'ar, donne une estimation optimale du nombre et des multiplicit\'es des fibres de Seifert  et du nombre et des torsions des espaces lenticulaires lorsque $X$ est une surface g\'eom\'etriquement rationnelle.

Lorsque $N$ admet une fibration de Seifert au dessus d'un orbifold $F$, nos r\'esultats  g\'en\'eralisent le th\'eor\`eme de Comessatti sur les surfaces rationnelles r\'eelles lisses~:
$F$ ne peut pas \^etre \`a la fois orientable et de type hyperbolique.

Nous montrons, ce qui est une surprise, que contrairement \`a ce qui se produit pour le th\'eor\`eme de Comessatti, il existe des exemples  o\`u $F$ est non orientable, de type hyperbolique, et $X$ est minimale.  
La technique employ\'ee est la construction d'une fibration de Seifert comme espace tangent projectivis\'e d'une surface Du Val. 

 \end{abstract}

\tableofcontents

\section*{INTRODUCTION}

Given a smooth real projective variety $W$ of dimension $n$, we
consider the topology
of a connected component $N$ of the set $W(\R)$ of its real points.

John Nash  proved in \cite{nash} that any compact connected
differentiable manifold $N$ is
obtained in this way, and went over to ask whether the same would hold if one
assumes $W$ to be geometrically rational.

However, when $W$ is a surface of negative Kodaira dimension, one is
able, after the
work of Comessatti
   \cite{Co14} for geometrically rational surfaces, to deduce drastical
restrictions for
the topology of $N$. Namely, if
$N$ is orientable, then it is diffeomorphic to a sphere or to a
torus: in other words,
$N$ cannot be
simultaneously oriented and
   of hyperbolic type.
In this note, we make a step towards a complete classification of the
topological types
for $N$ when
$W$ is a rationally connected 3-fold fibred by rational curves
(this is one of the
higher dimensional analogues of Comessatti's theorem).

This study was initiated by J\'anos Koll\'ar, in the third paper  \cite{KoIII}
of a ground-breaking series of articles applying the minimal model
program to the
study of the topology of real algebraic 3-folds.

 Koll\'ar's philosophy is that    a very important  condition  in order
to obtain restrictions upon the topological type of $W(\R)$
is that $W$ has terminal singularities and $K_W$  is Cartier along
$W(\R)$.

 Koll\'ar  proved in particular that if $W$ is a smooth
  3-fold  fibred by rational curves (in particular, $W$ has
negative Kodaira
dimension) and such that $W(\R)$ is orientable, then
 a connected component $N$ of $W(\R)$  is
essentially a Seifert
fibred 3-manifold or the connected sum of a finite number of lens
spaces. Note that in
\cite{hm1,hm2} it was shown that conversely all the above manifolds $N$  do
occur for some smooth 3-fold  $W$ fibred by rational curves. 

When
$W$ belongs to the subclass of rationally connected 3-folds fibred
by rational
curves, Koll\'ar proved some additional restrictions upon $N$ and
made three further
conjectures. In our first note
\cite{CM08} we proved two of the optimal estimates that Koll\'ar
conjectured to hold.
In the present note we prove the third estimate, which is the most
important one since it
allows us to conclude in particular that, if $N$ is a Seifert fibred
3-manifold, then
the base orbifold cannot be simultaneously oriented and
   of hyperbolic type.

Let us now introduce our results in more detail.  

Let $N$ be an oriented three
dimensional compact connected topological manifold without boundary. Take a
decomposition $N = N'
\#^a\mathbb{P}^3(\mathbb{R})\#^b(S^1\times S^2)$ with $a + b$ maximal
and observe that
this decomposition is unique by a theorem of Milnor \cite{Mil62}.

We shall focus our
attention on  the case where $N'$ is Seifert fibred or a connected sum of lens
spaces. We consider the  integers
$k : = k (N)$ and
$n_l := n_l(N)$,
$l = 1\dots k$  defined as follows:
\begin{enumerate}
\item if  $g \colon N' \to F$ is a Seifert fibration,  $k$ denotes
the number of
multiple fibres of
    $g$ and $n_1 \leq n_2 \leq \dots \leq  n_k$ denote the respective
multiplicities;
\item if $N'$ is a connected sum of lens spaces,  $k$ denotes the
number of lens spaces
and $n_1 \leq n_2 \leq \dots \leq n_k$, $ n_l \geq 3, \  \forall l$,
the orders of
the respective fundamental groups (thus we have a decomposition
$N' =
\#_{l=1}^k ( L (n_l, q_l)$ for some $1 < q_l < n_l$ relatively prime
to $n_l$).
\end{enumerate}

Observe that when $N'$ is a connected sum of lens spaces, the number
$k$ and the
numbers $n_l $, $l = 1, \dots k$ are well defined (again by Milnor's theorem).
In the case of a Seifert fibred manifold $N'$, these integers may a
priori depend upon
the choice of a Seifert fibration.

Three  results of our two notes are summarized by the following.

\begin{teo}\label{teo:multiple} Let $ W \to X$ be a real smooth projective
3-fold fibred by rational curves over
a geometrically rational \footnote{By \cite{ghs} these assumptions are
equivalent to:
$W$ rationally connected and fibred by rational curves.}
  surface
$X$. Suppose that
$W(\R)$ is orientable. Then, for each connected component $N
\subset W(\R)$, $k(N)  \leq 4$ and
$\sum_l (1-\frac 1{n_l(N)})\leq 2$. Furthermore, if $N'$ is Seifert fibred over
$S^1\times S^1$, then $k(N)=0$.
\end{teo}

This theorem answers, as we already said, some questions posed by Koll\'ar, see
\cite[Remark~1.2~(1,2,3)]{KoIII}.   In the first note, we proved the
estimate $k(N)
\leq 4$ and we showed that $k(N)=0$ if $N'$ is Seifert fibred over
the torus. Thus
Theorem~\ref{teo:multiple} follows from \cite[Corollary~0.2, and Theorem~0.3]{CM08} and from Theorem~\ref{teo:main} of the present paper using results of \cite{KoIII} as in \cite{CM08}. The
present note is mainly devoted to the proof of  the inequality
$\sum_l (1-\frac 1{n_l(N)})\leq 2$, see
Lemma~\ref{lem:question2}. 

The proof of this inequality goes as follows: let $W \to
X$ be a real
smooth projective
3-fold fibred by rational curves over a geometrically rational
surface $X$. Using the
same arguments as in \cite[Sec.~3]{CM08}, we reduce the proof of the
estimate for the
integers $n_l(N)$ to an inequality depending on the indices of
certain singular points
of a real component $M$ of the topological normalization of $X(\R)$ (see
Definition~\ref{def:norm}). In this process, the number $k(N)$ can be
made to correspond to the number of real singular points on $M$ which
are of type
$A^+_\mu$, and globally separating when $\mu$ is odd; each number $n_l(N)-1$
corresponds to the index $\mu_l$ of the singularity $A^+_{\mu_l}$ of
$M$. The main
part of the paper is devoted to the proof of the following.

\begin{teo}\label{teo:main} Let $X$ be a projective surface defined
over $\R$. Suppose
that $X$ is geometrically rational with Du Val singularities. Then a connected
component $M$ of the topological normalization $\overline {X(\R)}$
contains at most 4
singular points $x_l$ of type $A^+_{\mu_l}$  which are  globally
separating for $\mu_l$
odd. Furthermore, their indices
   satisfy
$$
\sum (1- \frac1{\mu_l +1} )\leq 2\;.
$$
\end{teo}

Let us now  give an interpretation of the above results in
terms of Geometric
Topology (see e.g.
\cite{Scott} for the basic definitions and classical results).
Suppose that $N'$ admits a Seifert
fibration with base orbifold $F$.
  From our main theorem \ref{teo:multiple} we infer that, if
the underlying manifold $\vert
F\vert$ is orientable, then the Euler characteristic of the
compact 2-dimensional orbifold $F$ is
nonnegative (see Proposition~\ref{proporbi}).
Thus, by the uniformization theorem for compact
2-dimensional orbifolds, $F$ admits a spherical structure or an
euclidean structure.

In general, a 3-manifold $N$ does not possess a geometric structure,
but, if it does, then the geometry
involved is unique.
Moreover, it turns out that every
Seifert fibred manifold admits a geometric structure.  The
geometry of $N$ is modeled on one of the six following models (see
\cite{Scott} for a detailed description of each geometry):
$$ S^3,S^2\times\R,E^3,\operatorname{Nil},\HH^2\times\R,
\widetilde{\operatorname{SL}_2\R},
$$ where $E^3$ is the $3$-dimensional euclidean space and $\HH^2$ is the
hyperbolic plane. The six
   above geometries are called the Seifert geometries.
The appropriate geometry for a Seifert
fibration is determined by  the Euler characteristic of the base orbifold
   and by the Euler
number of the Seifert bundle~\cite[Table~4.1]{Scott}.

Let $W$ be a real projective 3-fold  fibred by rational curves
and such that $W(\R)$ is
orientable, let $N\subset W(\R)$ be a connected component and let $N'$ be
the manifold defined as
above. Suppose moreover that $N'$ possesses a geometric structure.
By Theorem~\cite[Th.~1.1]{KoIII}, the geometry of $N'$ is one of the
six Seifert
   geometries.
Conversely, by \cite{hm1}, any orientable three dimensional manifold
endowed with any Seifert
geometry is diffeomorphic to a real component of a real projective
3-fold  fibred by rational
curves. But, when $W$ is rationally connected, the following
   corollary of our main theorem gives further restrictions.

\begin{cor}\label{cor:orbi} Let $W$ be a real smooth projective
rationally connected
3-fold fibred by rational curves. Suppose that $W(\R)$ is orientable
   and let $N$ be a connected
component of $W(\R)$. Then neither $N$ nor $N'$ can be endowed with a
$\widetilde{\operatorname{SL}_2\R}$ structure or with a $\HH^2\times
\R$ structure
whose base orbifold $F$  is orientable.
\end{cor}

Observe moreover that in \cite{KoIII} all compact 3-manifolds
with $S^3$ or $E^3$ geometry,
and some manifolds with $\operatorname{Nil}$ geometry, are realized as a real
component of a real smooth projective rationally connected 3-fold
fibred by rational
curves.

There remains of course the question about what happens when
$N$ is Seifert fibred over a nonorientable orbifold $F$: is the
orbifold still not of hyperbolic type?
In the last section we show that the answer to this question
is negative. We produce indeed an
example of a smooth 3-fold $W$, fibred by
rational curves over a  Du Val Del Pezzo surface $X$, where
$W(\R)$ is  orientable,
and contains a connected component which is Seifert fibred over a
nonorientable base orbifold  of hyperbolic type.

The striking fact is here that $X$ is a real minimal surface: this contrasts
with Comessatti's theorem: since indeed a real minimal nonsingular geometrically rational
surface cannot have an component which is of hyperbolic type.

\begin{teo}\label{hyperbolic} There exists a minimal real Du Val Del Pezzo surface
$X$ of degree $1$ having exactly two singular points, of type $A_2^+$,
and such that the real part $X(\R)$ has a
connected component containing the two singular points and which is
homeomorphic to a  real projective plane.

Let $W'$ be the projectivized tangent bundle of $X$: then
$W'$ has terminal singularities, $ W' (\R)$ is contained 
in the smooth locus of $W'$, 
in particular if $W$ is obtained  resolving  the singular points of $W'$,
then $ W (\R) =  W' (\R)$.

Moreover $ W (\R)$ is  orientable and  contains a connected component $N$
which is  Seifert fibred over a non
orientable orbifold of hyperbolic type (the real projective plane with two
points of multiplicity 3).
\end{teo}

Briefly, now, the contents of the paper.

Sections~\ref{sec:recall} and~\ref{sec:sevenconfig} are devoted to
the reduction of the
   proof of the main theorem to the assertion of non existence of seven
configurations of
singular points on a real component of a Du Val Del Pezzo surface of degree~1.

Two main methods used here are borrowed from \cite{CM08}: namely, the
generalization
of Brusotti's theorem to the effect that one can independently take
any smoothing of the
singularities of a Du Val Del Pezzo surface, and also the 
use of the plane model
where the family of hyperplane
sections of the quadric cone $Q$ is represented by
the family of parabolae in the plane with
a fixed asymptotic
direction.
These methods combine with a delicate argument, suggested by E.
Brugalle, excluding
the possibility of an  intersection  of $Q$ with a cubic surface yielding an
irreducible curve $B$ with four real cusps (see \ref{4cusps}).

Section~\ref{sec:config} introduces the main tools used in the proof
(the topological
classification of real smooth Del Pezzo surfaces of degree 1, and the
choice of the
appropriate partial smoothings), and ends with the exclusion of two
configurations
via complicated although elementary topological considerations.

Section~\ref{sec:euler} uses a  classification of  critical points
for the projection
of $B$ and a precise table for the local contributions to the
multiplicity of the
discriminant and for the  local contribution to the Euler number
in order to exclude  two more
cases.

Section~\ref{sec:lasteuler} proves Theorem~\ref{teo:main} by
excluding the three
remaining cases by combining all the previous tools with an ad hoc analysis and
with two new
tools, namely: the use of the Comessatti characteristic, relating the
total Betti number of
the real part with the one of the complex part,  and the calculation
of the contributions of the
singularities to the Picard and to the various Euler numbers.

Finally, Section~\ref{sec:3fold} is devoted to the proof of
Lemma~\ref{lem:question2} and  in Section~\ref{orbi}, after showing that
   the base orbifold cannot be oriented and hyperbolic, we exhibit 
the example of a projectivized tangent bundle over a Du Val Del Pezzo surface
for which a component $N$ is Seifert fibred with base orbifold of
hyperbolic type.

In the course of this complicated construction we give a quite general method
to construct Seifert fibrations as projectivized tangent bundles of
surfaces with $A_n$- singularities.

We want to thank E.~Brugalle
for pointing out the statement of lemma \ref{4cusps} and suggesting
the main idea of the proof , and Ingrid Bauer for helping us to understand the
configuration of lines on Del  Pezzo surfaces of degree $1$.

\section{Singular geometrically rational surfaces}\label{sec:recall}

Using the results and notation of \cite[Section 1]{CM08}, we reduce
the proof of
Theorem~\ref{teo:main} to the proof of a statement about singular Del
Pezzo surfaces of
degree 1 with small Picard number $\rho$.

Recall that  a surface singularity which is  a rational double point
is also called a Du Val
singularity and that a projective surface $X$ is called a {\em Du
Val} surface if  $X$ has
only Du Val singularities. A surface singularity is {\em of type
$A^+_\mu$} if it is real analytically equivalent to
$ x^2 + y^2 - z^{\mu + 1} = 0,\ \mu \geq 1\;;
$ and {\em of type $A^-_\mu$} if it is equivalent to
$ x^2 - y^2 - z^{\mu + 1} = 0,\ \mu \geq 1
$. The type $A_1^+$ is real analytically isomorphic to $A_1^-$;
otherwise, singularities
with different names are not isomorphic.

We recall some definitions due to Koll\'ar (see \cite[Section 1]{CM08}).
\begin{df}\label{def:norm} Let $V$ be a simplicial complex with only
a finite number of
points
$x \in V$ where $V$ is not a manifold. Define the {\em topological
normalization}
$$
\overline{n} \colon \overline{V} \to V
$$  as the unique proper continuous map such that $\overline{n}$ is a
homeomorphism
over the set of points where $V$ is a manifold and $\overline{n}^{-1}(x)$ is in
one-to-one  correspondence with the connected components of a good punctured
neighborhood of $x$ in $V$ otherwise.
\end{df}

Observe that if $V$ is pure of dimension $2$, then $\overline{V}$
is a topological manifold (since each point of $\overline{V}$ has
a neighbourhood which is a cone over $S^1$).

\begin{df} Let $X$ be a real Du Val surface, and let
$x \in X(\mathbb{R})$ be a singular point of type $A^\pm_\mu$ with
$\mu$ odd. The
topological normalization
$\overline{X(\mathbb{R})}$ has two connected components locally near
$x$. We will say
that
$x$ is {\em globally separating} if  these two local components lie
on different
connected components of $\overline{X(\mathbb{R})}$ and {\em globally
nonseparating}
otherwise. Let
\begin{multline*}
\PPP_X := \operatorname{Sing}X \setminus \left\{ x \textrm{ of type }
A^-_\mu,\ \mu
\textrm{ even}\right\} \\
\setminus \left\{ x \textrm{ of type } A^-_\mu,\ \mu \textrm{ odd and
} x \textrm{ is
globally nonseparating}\right\} \;.
\end{multline*}
\end{df}

Let $X$ be a real Du Val surface, let
$\overline{n} \colon \overline{X(\R)} \to X(\R)$ be the topological
normalization, and
let
$M_1,M_2,\dots, M_r$ be the connected components of  $\overline{X(\R)}$. By
\cite[Cor.~9.7]{KoIII}, the unordered sequence of numbers 
$$
m_i :=
\#(\overline{n}^{-1}(\PPP_X) \cap M_i ),\ i = 1,2,\dots, r
$$
is an
invariant for
extremal birational contractions of Du Val surfaces.

We will now reduce the proof of Theorem~\ref{teo:main} to the proof
of the following.

\begin{teo}\label{teo:dps1} Let $X$ be a real Du Val Del Pezzo
surface of degree 1 with
$\rho(X) \leq 2$.  Then  $m_i \leq 4$, $i = 1,2,\dots, r$, and
moreover for any $M
:=M_i$ such that $\overline{n} (M)$ contains $A^+_{\mu_1} +
A^+_{\mu_2} + \dots +
A^+_{\mu_{m_i}}$ where $A^+_{\mu_l}$ is globally separating for
$\mu_l$  odd, we have:
$$
\sum_{l = 1}^{m_i} (1- \frac1{\mu_l +1} )\leq 2 \;.
$$
\end{teo}

Up to Section~\ref{sec:3fold}, the sequel of this paper is devoted to
the proof of
Theorem~\ref{teo:dps1}.

\section{Reducing to seven configurations}\label{sec:sevenconfig}

Numerically, the following configurations of $A_{\mu}^+$ 
singularities are the only ones
allowed by the inequality
\begin{equation}
\sum_{l = 1}^{m_i} (1- \frac1{\mu_l +1} )\leq 2 \;.\label{eq:main}
\end{equation}

\begin{itemize}
\item $m_i = 4$ and the configuration is $4A_1^+$,
\item $m_i = 3$ and the configuration is
\begin{itemize}
\item $2A_1^++A_\mu^+$, any $\mu$, or
\item $A_1^+ + A_2^+ + A_\mu^+$, $\mu \leq 5$, or
\item $A_1^+ + 2A_3^+$,
\item $3A_2^+$,
\end{itemize}
\item $m_i = 2$.
\end{itemize}

Recall that  a Du Val Del Pezzo surface $X$ is by definition a Du Val
surface (i.e., a
surface with only rational double points as singularities)  whose
anticanonical divisor
is ample, see \cite[Section~2]{CM08}.  The anticanonical model of a
Del Pezzo surface
$X$ of degree 1 is a ramified double covering $q \colon X \to Q$ of a
quadric cone $Q
\subset
\PP^3$ whose branch locus is the union of the vertex of the cone with
a curve $B$
not passing through the vertex and which is the complete intersection
of the cone with a cubic surface.

Let $X$ be a real Du Val Del Pezzo surface of degree 1 and let
$X'$ be the singular surface obtained from $X$ by blowing up the
pull-back by $q$ of
the vertex of the cone (which is a smooth point of $X$). The surface
$X'$ is a ramified
double covering of the Hirzebruch surface $\F_2$ whose branch curve
is the union of the
unique section of negative selfintersection, the section at infinity
$\Sigma_\infty$,
and the trisection $B$ of the ruling $p \colon \F_2 \to \PP^1$, which
is disjoint from
$\Sigma_\infty$.  The composition $X' \to \F_2 \to \PP^1$ is a real
elliptic fibration.

The different cases that we shall now consider are distinguished by
the number of
irreducible components of the trisection $B$. Notice that if all the
singular points
are of type $A_1$, the conclusion of  Theorem~\ref{teo:dps1} follows from
\cite[Proposition~2.1]{CM08}.

\subsection{Three components}

If $B$ has strictly more than 4 real singular points, all the possible cases
are enumerated in
\cite[Section~2]{CM08}, and an inspection of [Ibid., Figures 1, 2, 3]
shows that for
any connected component of the  complement $\F_2 (\R) \setminus B (\R)$, the
configuration is $4 A_1$ or $A_3+2A_1$. Thus the inequality
(\ref{eq:main}) holds
except possibly in the situation where two irreducible components
of $B$ are tangent to the third one. It
turns out that there is only one normal form for this situation, see
Figure~\ref{fig:2tacnodes}. Indeed, the affine part of $B$ is a union of three
parabolae and without loss of generality, these three parabolae are
given by $y = 0$,
$y = x^2$ and
$ y = \alpha (x - a)^2,\ \alpha,\ a \in \R$ [Ibid.]. We have $a\ne 0$,
else $B$ has a triple point with an infinitely near triple point,
contradicting the fact that $X$ has only Du Val singularities.
Furthermore, in order to get at least three real intersection
points, $\alpha$ has to be positive. Up to reflection $x
\leftrightarrow -x$, this
leads to one possibility.

\begin{figure}[htbp]
\begin{center}
     \epsfysize=3cm
    \epsfbox{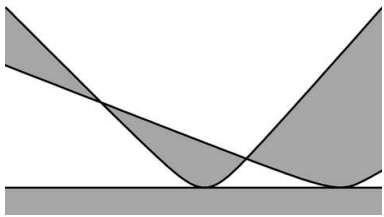}
\caption{Three parabolae with two tacnodes.}\label{fig:2tacnodes}
\end{center}
\end{figure}

Recalling that in this figure two components are connected at infinity if their
boundaries have two unbounded arcs belonging to the same pair of
parabolae, we see that
none of the connected components of $\F_2 (\R) \setminus B (\R)$
contains more than 3
singular points and at most two of them are tacnodes. Thus
(\ref{eq:main}) holds also
in this case.

\subsection{Two components} Then  $B = L \cup C$ where $C$ is a
bisection of the
ruling $p$ and $L$ is a section. The bisection $C$ has arithmetic
genus one, hence it has at most one
double point $A_1$
or $A_2$ and at most 4 intersection points with the section $L$.

If $C$ is non singular, we have $4A_1$ or $2A_1+A_3$ or only two
singular points. In
each case we get an allowed  configuration.

Assume $C$ is singular: if $B$ has 5 singular points,  we are done
 since either
all singular points are  of type $A_1$,  see \cite[Figures 4,
5, 6]{CM08}, or we are in the situation depicted in 
\cite[Figure 7]{CM08} and then the $A_2^+$ is on a component
with only two other singularities, of type  $A_1$. If $B$ has 4
singular points, the possibilities are
$A_1+A_3+2A_1$, or $A_2+A_3+2A_1$. If $B$ has 3 singular points, the
possibilities are
$A_1+2A_3$, or $A_2+2A_3$, or $A_2+A_1+A_5$.

Thus if $B$ has two irreducible components, we get the conclusion of
Theorem~\ref{teo:dps1} unless the configuration of singular points is
$A_3 + 3A_1$,
or $A_3 + A_2 + 2A_1$, or $2A_3 + A_2$.

\subsection{One component} If the trisection is irreducible, then it
has at most 4
singular points, since $B(\C)$ has genus 4.

\begin{lem}\label{4cusps} The real curve $B$ cannot have 4 real cusps.
\end{lem}
\begin{proof} Suppose that $B$ is irreducible with 4 real
cusps. Choose three of them. Let $L'$ be a section of the ruling $p$
corresponding to a
plane section of $Q$ passing through
these three points; for an appropriate choice of the plane model of $Q$
(see \cite{CM08}, beginning of section 2) we may assume
$L'$ to be the horizontal $x$-axis $ y = 0$ in the plane

Since the intersection number $ L' \cdot B = 6$, we get that $L'$ intersects
$B$ exactly at the three chosen cusps, and transversally. This means that,
w.l.o.g., $B$ lies in the upper halfplane: in fact, since $B$ is
rational and irreducible,
then its real part $B(\R)$ is homeomorphic to $S^1$, in
particular it is  connected.

Observe moreover that none of the cusps is tangent to a fibre, since each cusp
gives a contribution at least 3 to the local multiplicity of the
discriminant of $B$, and this contribution becomes 4 if the cusp is tangent to
the fibre: and the order of the discriminant is 12.

   In fact, we get more from this calculation: the projection $p$ has
no further critical
points on $B$.

It follows that the projective line with coordinate $x$ is divided
into 4 open intervals,
such that the cardinality of the fibre of $p \colon B(\R) \ra
\PP^1_{\R}(\R)$ varies
alternatingly from $3$  to $1$.

On the intervals where we have  3 counterimages, it makes sense to
talk about first,
second and third branch (ordered according increasing value of the
$y$-coordinate), on
each interval it makes sense to talk about the highest and the lowest branch.

Whenever one moves on $\PP^1_{\R}(\R)$ and goes across a cusp lying on
the $x$-axis, the highest branch continues to be the highest
branch.

Since three of the cusps lie on the x-axis, we may assume that the fourth cusp
is located at $ x = \infty$, and the three cusps with $ y =0$ occur for
$ x = A,B,C$ where $ A < B < C$. Then the highest branch over the interval
$ (- \infty , A)$ remains the highest branch on the whole real line
by virtue of
the previous remark. By compactness of $B(\R)$ we get a connected component
of $B(\R)$ mapping to $\PP^1_{\R}(\R)$ homeomorphically,
contradicting our previous
assertion about the cardinalities of the fibres.

\end{proof}

Thus, if $B$ is irreducible,  we observe that $B$ has arithmetic genus $4$,
and nonnegative geometric genus: hence the `number of double points'
$\delta$  is
at most $4$.
But  each point of type
$A_n$  contributes exactly $ [ \frac{n+1}{2}]$ double points.
    Therefore an elementary calculation shows that we get the
conclusion of  Theorem~\ref{teo:dps1} unless the configuration of
singular points is
one of the following : $A_4 + 2A_2$, $A_3 + 2A_2$, or
$3A_2 + A_1$, or $2A_2 + 2A_1$, or $A_2 + 3A_1$.

We are going now to exclude the first case by an argument  similar 
to the one of Lemma \ref{4cusps}, even if it could also be treated by
the same methods used in Section ~\ref{sec:euler}.

\begin{lem}\label{2cusps} The real curve $B$ cannot have 2 real
$A_2$ singularities and an $A_4$ singularity.
\end{lem}
\begin{proof} We already know that $B$ is irreducible and we argue as in
Lemma \ref{4cusps}, assuming that the three singular points lie on
the horizontal $x$-axis $ \{y = 0 \} := L'$ in the plane and that,
since  $B(\R)$ is homeomorphic to $S^1$, $B$
and  $L'$ intersect  exactly at the three chosen points, and transversally,
hence $B(\R)$  lies in the upper halfplane.

If none of the cusps is tangent to a fibre, since each cusp $A_{2n}$
gives a contribution $2n+1$ to the local multiplicity of the
discriminant of $B$,  and the order of the discriminant is 12,
there is exactly another  critical point for the restriction of
the projection $p$ to $B$, and the same argument as in Lemma
\ref{4cusps} provides the same contradiction.

   There remains the case where exactly one cusp is vertical,
and there are no further critical points.

It follows that the projective line with coordinate $x$ is divided
into 3 open intervals,
and  the cardinality of the fibre of $p \colon B(\R) \ra
\PP^1_{\R}(\R)$ must be
equal to $1$ on the two intervals neighbouring the vertical cusp.
At the two other cusps the highest branch remains the highest,
and we get the usual contradiction (since over  the third interval we
have three
branches).

\end{proof}

In any case, regardless of the difference between $A_\mu^+$ and
$A_\mu^-$, we have
reduced the problem to the exclusion of 7 configurations. For any of these
configurations, we can suppose that all singular points are of type
$A_\mu^+$ with
$A^+_{\mu_l}$ globally separating for $\mu_l$ is odd. Indeed, if one
of the point is
not of this type, the sum $\sum (1- \frac1{\mu_l +1} )$ restricted to
the remaining
points if less than or equal to 2.

Summarizing, we get seven remaining configurations to be excluded:

\begin{enumerate}
   \item\label{2A3+A2}
$2A^+_3 + A_2^+$ (Section~\ref{sec:lasteuler})
\item\label{A3+2A2}
$A^+_3 + 2A_2^+$ (Section~\ref{sec:euler})
\item\label{A3+A2+2A1}
$A^+_3 + A_2^+ + 2A_1^+$ (Section~\ref{sec:config})
\item\label{A3+3A1}
$A^+_3 + 3A_1^+$ (Section~\ref{sec:config})
\item\label{3A2+A1}
$3A^+_2 + A_1^+$ (Section~\ref{sec:euler})
\item\label{2A2+2A1}
$2A^+_2 + 2A_1^+$ (Section~\ref{sec:lasteuler})
\item\label{A2+3A1}
$A^+_2 + 3A_1^+$  (Section~\ref{sec:lasteuler})
\end{enumerate}

\section{Smoothings of Du Val  Del Pezzo surfaces}\label{sec:config}

We recall that our problem consists in giving an estimate concerning
   the configurations of
certain singular points lying on a component of the topological normalization
of a real Du Val Del Pezzo surface $X$. For this purpose, we want to understand
as much as possible the topology of  $X (\R )$, and we do this by
taking a global
smoothing of $X$, and then using the known topological classification
of smooth real Del Pezzo surfaces of degree 1.

   The best strategy is to choose a global smoothing realizing certain
local smoothings of the singularities chosen a priori. That this can be done
for all choices of the local smoothings holds true by a generalization of
the theorem of Brusotti which was proven in our preceding paper.

\begin{teo}\cite[Th.~4.3]{CM08}\label{teo:generalizedbrusotti} Let
$X$ be a Du Val Del
Pezzo surface. One can obtain, by a global small deformation of $X$,
all the possible
local smoothings of the singular points of $X$.
\end{teo}

\begin{prop}[Global]\label{sec:dps1} Let $X$ be a real smooth Del
Pezzo surfaces of
degree 1: then the real part $X(\R)$ is diffeomorphic to one of the
surfaces in the
following list:

\begin{itemize}
\item $\mathbb{P}^2(\R) \sqcup p \SSS,\  p=1,\dots,4$;
\item $\mathbb{P}^2(\R) \sqcup \KKK$;
\item $\#^{3}\mathbb{P}^2(\R) \sqcup  \SSS$;
\item $\#^{2p+1}\mathbb{P}^2(\R),\  p=0,\dots,4$.
\end{itemize}

Here $\#^l\mathbb{P}^2(\R)$ denotes the connected sum of $l$ copies
of the real
projective plane, $\KKK = \mathbb{P}^2(\R)\# \mathbb{P}^2(\R)$
denotes the Klein bottle
and $p\SSS$ denotes the disjoint union of $p$ copies of the $2$-sphere.
\end{prop}

\begin{proof} It is the well-known classification of real smooth Del
Pezzo surfaces,
see e.g. \cite{DIK}.
\end{proof}

\begin{lem}[Local] Consider a real singular point of a surface $X$ of type
$A^+_\mu$, of local equation
$z^2 = f (x,y)$ where $f$ vanishes at the origin. Then for each case
$\mu\in\{1,2,3\}$, there exists local smoothings $X_{\varepsilon}$
with equation $ z^2
= f_\varepsilon (x,y)$, such that $X_{\varepsilon}(\R)$ is
represented by one of the
Figures~\ref{fig:a1}, \ref{fig:a2}, or \ref{fig:a3}.

\begin{figure}[htbp]
\begin{center}
     \epsfysize=2.5cm
    \epsfbox{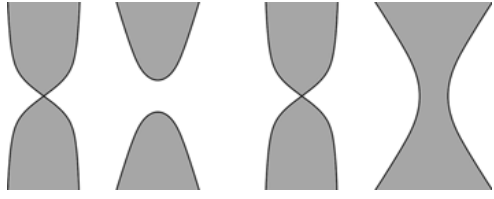}
\caption{The {\em cut} and the {\em cylinder} smoothings  of the node
$A^+_1$.}\label{fig:a1}
\end{center}
\end{figure}

\begin{figure}[htbp]
\begin{center}
     \epsfysize=2.5cm
    \epsfbox{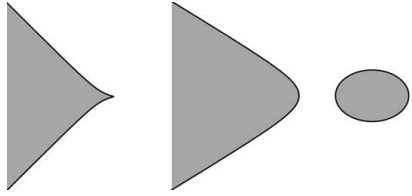}
\caption{The {\em +sphere} smoothing  of the cusp $A^+_2$.}\label{fig:a2}
\end{center}
\end{figure}

\begin{figure}[htbp]
\begin{center}
     \epsfysize=2.5cm
    \epsfbox{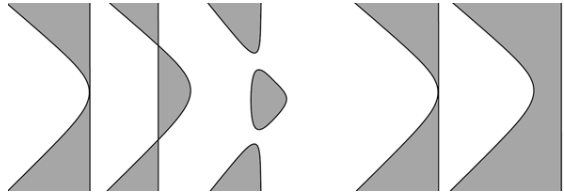}
\caption{The {\em cut+sphere} and the {\em cylinder} smoothings  of the tacnode
$A^+_3$.}\label{fig:a3}
\end{center}
\end{figure}

\end{lem}

\subsection{Topology of connected components}

Let $X$ be a real Du Val Del Pezzo surface of degree 1. Recall that
$X'$ denotes the
singular elliptic surface obtained from $X$ by blowing up a smooth
point. We denote by
$\overline{n} \colon \overline {X'(\R)} \to X'(\R)$ the topological
normalization of
the real part and we assume that there is a connected component $M_0$ of
$\overline{X'(\R)}$ whose image by $\overline{n}$ contains at least
three singular
points of $X'$. Furthermore, we assume that the singular points are
only of type
$A_\mu^+$, with $A^+_\mu$ globally separating for $\mu$ odd.

Let $M_j$, $j = 1,\dots c$ be the other components of
$\overline{X'(\R)}$ such that
$\overline{n}(M_j)$  and  $\overline{n}(M_0)$ intersect (in some singular
point of $X'(\R)$). Any singular point $A^+_\mu$ of $M_0$ with $\mu$ odd
is globally separating,
while the ones with
$\mu$ even are not, thus in particular the number
$c$ satisfies $1 \leq c \leq \#\{P  \in \overline{n}(M_0)| 
P {\it of \  type \ } A_\mu^+,\ \mu \textrm{
odd}\}$.

   Let's denote by $M_\infty$ the connected component of
$\overline{X'(\R)}$ which meets
the section at infinity, i.e.,
$\overline{n}(M_\infty)\bigcap\Sigma_\infty \ne \emptyset$. In the
proof of the main
theorem we will often use the distinction between the cases $M_\infty
= M_0$ and
$M_\infty = M_j$ for some $j \ne 0$.

\begin{lem}\label{lem:klein} The component $M_\infty \subset
\overline{X'(\R)}$ of the topological normalization is a Klein bottle
unless the
elliptic fibration has two white returns (see Table~\ref{tab:euler}).
In the latter
situation, $\overline{X'(\R)}$ contains at most another component
which  is then a
sphere.
\end{lem}

\begin{proof}

If the fibre of the double covering $q' \colon X' \to \F_2$ over a real point
$P$  contains a real point, we shall say that $P$ belongs to the
region of positivity,
which we denote by $\F_{2+}$.

The section $\Sigma_\infty$ is part of the branch locus and is
bilateral in $\F_2$.

Consider $U := \F_2\setminus \Sigma_\infty$ which is an oriented
$\bbA^1$-bundle over $\PP^1$, and indeed homeomorphic to $\PP^1_{\R} \times
\R$. Hence we  take corresponding coordinates $(x,y) \in \PP^1_{\R} \times
\R$ for the points of $U$.

We may assume without loss of generality that $(x,y)
\in \F_{2+}$ for $y >> 0$.

Consider now the function $\eta \colon \PP^1(\R) \to
\R$, $\eta(x) : = \inf \{y \  |  \ \{x\}\times [y,\infty[ \subset \F_{2+} \}$.
Therefore,  if
$\eta$ is a continuous function on $\PP^1(\R)$, then we have $M_\infty =
\KKK$.

For further use, we notice that:
\begin{lem}\label{lem:monoton} Let $M \subset \overline{X'(\R)}$ be a connected
component of the topological normalization of $X'(\R)$, and consider
$x$ as a function on the boundary of $M$: then the number of changes
of monotonicity of
$x$ is even.
\end{lem}

Let $\Delta(x)$ be the discriminant of the elliptic fibre over $x$ (i.e., the
discriminant of the degree three polynomial in $y$ whose zero set is
the trisection).
In view of Table~\ref{tab:euler} (go two pages ahead), the only root of
$ \Delta(x)$ which can break the continuity of $\eta$ corresponds to
a white return
($A_0,e=-1$).

If the lower part of the white return branch continues and meets as
first critical
point of $p\colon B(\R) \to \PP^1(\R)$ a point which contributes one change of
monotonicity of $x$  (that is, a black return, or a black node, or a
tangent node, or a
transversal cusp or tacnode), then we can topologically deform to the
case where
$\eta$ is continuous. The flex and the tangent cusp are clearly
irrelevant and if the
first met critical point is a white node, we can perform a cut
smoothing and pass to
the next critical point. The only obstacle is then the case when we
meet another white
return singularity on the branch curve. In this case, one 
sees easily that there is another component $D$ of $B(\R)$ disjoint
from the   component $D'$ containing the white return branches,
hence $b_1(M_\infty)\geq 4$ (take the 4 cycles respective inverse
images of the section  at infinity $\Sigma_{\infty}$, of $D$, and 
 of two segments, one joining $D'$ with $\Sigma_{\infty}$, the other
joining $D'$ with $D$). Recall
that  the topological normalization
$\overline{X(\R)}$
of the real Del Pezzo surface $X$ can be realised by a global
smoothing of $X$, see
\cite[Lemma~4.4 and Theorem~4.3]{CM08}.  Thus the component of
$\overline{X(\R)}$
corresponding to $M_\infty$ has $b_1 \geq 3$ and, by \ref{sec:dps1},
either we have
$\overline{X(\R)} =
\#^{3}\mathbb{P}^2(\R) \sqcup  \SSS$ or $\overline{X(\R)} =
\#^{2p+1}\mathbb{P}^2(\R)$ for some $p = 1,\dots,4$.
\end{proof}

\begin{rem} More generally, by \ref{sec:dps1}, the real part of any
global smoothing
$X'_\varepsilon$ of $X'$, including the case when
$X'_\varepsilon(\R) = \overline{X'(\R)}$, is diffeomorphic to one of
the surfaces in
the following list,
\begin{itemize}
\item $\KKK \sqcup p \SSS,\  p=1,\dots,4$;
\item $\KKK \sqcup \KKK$;
\item $\KKK\#\KKK \sqcup  \SSS$;
\item $\#^{q}\KKK,\  q = 1,\dots,5$.
\end{itemize}
\end{rem}

\begin{lem}\label{topotype}  Let $X$ be a real Du Val Del Pezzo
surface of degree one,
and let $X'$ be the corresponding rational elliptic surface. Let
$X'_\varepsilon$ be a global smoothing of $X'$. Then we have the
following estimates
for the  Betti numbers $$ b_i ( X'_\varepsilon(\R) ): =  rank  H_i
(X'_\varepsilon(\R)
, \Z/2)$$
\begin{itemize}
\item $b_0(X'_\varepsilon(\R))\geq 3 \Rightarrow b_1(X'_\varepsilon(\R)) = 2$.
\item $b_0(X'_\varepsilon(\R))\geq 2 \Rightarrow
b_1(X'_\varepsilon(\R)) \leq 4$.
\item In any case, $b_0(X'_\varepsilon(\R))\leq 5$.
\end{itemize}
\end{lem}

\subsection{Exclusion of $A^+_3 + A_2^+ + 2A_1^+$ and $A^+_3 + 3A_1^+$}

For each node of $\overline{n}(M_0)$ connecting $\overline{n}(M_0)$ with some
$\overline{n}(M_j)$, we choose the cut smoothing if this point is the
only singular
point on $\overline{n}(M_j)$. Otherwise, we choose the cylinder
smoothing. We do the
+sphere smoothing for the cusp. For the tacnode, connecting
$\overline{n}(M_0)$ with
some $\overline{n}(M_j)$, we choose the cut+sphere smoothing if this
point is the only
singular point on
$\overline{n}(M_j)$ or if we are in the last two cases in the next
list. Otherwise, we
choose the cylinder smoothing.  Recalling that
$b_1\left(\overline{X'(\R)}\right) \geq 2$, we obtain the following
inequalities for
the Betti numbers of $X'_\varepsilon(\R)$. The different cases are
distinguished by the
number $c$ defined above.

\begin{itemize}
\item[(\ref{A3+A2+2A1})]
$A^+_3 + A_2^+ + 2A_1^+$
\begin{itemize}
\item[$c = 3$:] $b_0 \geq 6$;
\item[$c = 2$:] $b_0 \geq 3$, and $b_1 \geq 4$;
\item[$c = 1$:] $b_0 \geq 2$, and $b_1 \geq 6$.
\end{itemize}

\medskip
\item[(\ref{A3+3A1})]
$A^+_3 + 3A_1^+$
\begin{itemize}
\item[$c = 4$:] $b_0 \geq 6$;
\item[$c = 3$:] $b_0 \geq 3$, and $b_1 \geq 4$;
\item[$c = 2$:] (cut+sphere smoothing for the tacnode) $b_0 \geq 3$, and
$b_1 \geq 4$.
\item[$c = 1$:] (cut+sphere smoothing for the tacnode) $b_0 \geq 2$, and
$b_1 \geq 6$.
\end{itemize}

\end{itemize}

In each case, these inequalities  contradict Lemma~\ref{topotype}. Thus cases
(\ref{A3+A2+2A1}) and (\ref{A3+3A1}) are excluded.

\section{The Euler number on an elliptic fibration}\label{sec:euler}

Recall that $X'$ is a singular surface obtained from the singular
degree 1 Del Pezzo
surface $X$ by blowing up a smooth point. It is a ramified double
covering of the
Hirzebruch surface $\F_2$ whose branch locus is the union
$\Sigma_\infty \cup B$ where
$B$ is a trisection of the ruling $p \colon \F_2 \to \PP^1$. The
composition $X' \to
\F_2 \to \PP^1$ is a real elliptic fibration and $\Delta(x)$ denotes
the discriminant
of the elliptic fibre over $x$ (i.e., the discriminant of the degree
three polynomial
in $y$ whose zero set is the trisection). Table~\ref{tab:euler} gives a
local topological
description of the fibration over a neighbourhood of a real zero of $\Delta$,
   in terms of two basic numerical invariants, namely the multiplicity
of the zero of
$\Delta$, and the Euler number of the real part of the
singular fibre of the elliptic surface.
The table considers only the singular points that we have to deal
with, and introduces a
name for each case, which will be used in the course of the forthcoming proofs.
Observe finally that, in drawing as black the region of positivity,
we have used the
convention introduced in Lemma \ref{lem:klein}. Finally, a point of type
$A_0$ is here a smooth point of $B$ which is a critical point for the
restriction of $p$ to
$B$.


\begin{table}
\setlength{\doublerulesep}{\arrayrulewidth}
\begin{tabular}{|| c | c | c | c | c ||}
\hline\hline
\rule[-8pt]{0pt}{25pt} Type & Fibre type & Picture & Multiplicity of
$\Delta$ & Euler
number $e$ \\
\hline
\rule[-8pt]{0pt}{45pt} $A_0$ &  black return &   \epsfysize=32pt
\epsfbox{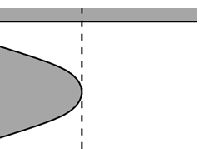}  & $1$ & $1$  \\
\hline
\rule[-8pt]{0pt}{45pt} $A_0$ & white return  &  \epsfysize=32pt
\epsfbox{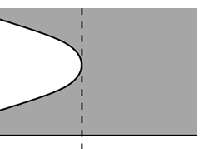} & $1$ & $- 1$ \\
\hline
\rule[-8pt]{0pt}{45pt} $A_0$ & flex &  \epsfysize=32pt \epsfbox{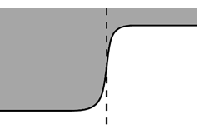}  &
$2$  & $0$  \\
\hline
\rule[-8pt]{0pt}{45pt} $A_1^+$ & black node & \epsfysize=32pt
\epsfbox{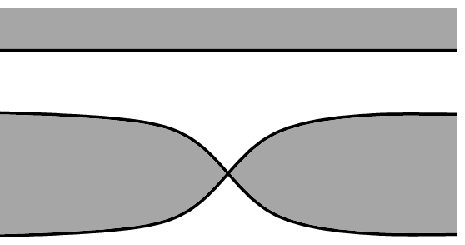} & $2$ & $1$  \\
\hline
\rule[-8pt]{0pt}{45pt} $A_1^+$ & white node & \epsfysize=32pt
\epsfbox{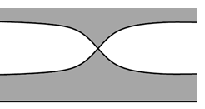} & $2$ & $-1$  \\
\hline
\rule[-8pt]{0pt}{45pt} $A_1^+$ & tangent node & \epsfysize=32pt
\epsfbox{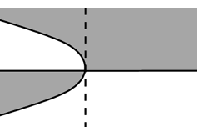} & $3$ & $0$  \\
\hline
\rule[-8pt]{0pt}{45pt} $A_2^+$ & transversal cusp & \epsfysize=32pt
\epsfbox{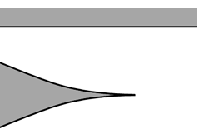} & $3$ & $1$  \\
\hline
\rule[-8pt]{0pt}{45pt} $A_2^+$ & tangent cusp & \epsfysize=32pt
\epsfbox{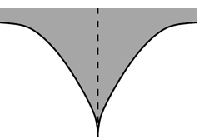} & $4$ & $0$  \\
\hline
\rule[-8pt]{0pt}{45pt} $A_3^+$ & tacnode & \epsfysize=30pt
\epsfbox{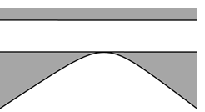} & $4$ & $1$\\
\hline\hline
\end{tabular}
\par\medskip
\caption{Singular points of the elliptic fibration
and contributions to
the Euler number.}\label{tab:euler}
\end{table}
\bigskip

We use now Table~\ref{tab:euler} in order to proceed with our case by case
exclusion.
\subsection{Exclusion of $A^+_3 + 2A_2^+$ and $3A^+_2 + A_1^+$}

\begin{itemize}

\item[(\ref{A3+2A2})]
$A^+_3 + 2A_2^+$,

Here $c=1$, and $\overline{n}(M_1),\overline{n}(M_0)$ meet in the tacnode
$A^+_3$. By doing the +sphere smoothing for each cusp and the
cut+sphere smoothing for
the tacnode, we get at least 5 connected components, hence a Klein
bottle and 4 spheres
by \ref{topotype}, and $\overline{X'(\R)} = \KKK
\sqcup \SSS$. We conclude that $e(X'(\R)) = 1$.

If $M_0 \ne M_\infty$ (thus $M_0=\SSS$ and $M_1=\KKK$), the cusps are
transversal.
Since the total multiplicity of $\Delta$ is 12, the fibration has at
most two real $A_0$
singular fibers. But any white return stays on
$\overline{n}(M_\infty)$.  Then on the
boundary of   $\overline{n}(M_0)$  the monotonicity of the function
$x$ has three
changes, a contradiction.

If $M_0 = M_\infty$, then the component $M_1=\SSS$ must contain a
black return (recall
that the two cusps belong to $\overline{n}(M_0)$). The existence of
the tacnode on
$\overline{n}(M_\infty)$ forces a white return. The contributions to
the multiplicity
of $\Delta$ are then $4 + 3+3+1+1=12$ which implies that the two
cusps are transversal.
The contributions to the Euler characteristic of $X'(\R)$ are then $1
+1+1+1-1=3$, a
contradiction.

Thus case (\ref{A3+2A2}) is excluded.

\medskip
\item[(\ref{3A2+A1})]
$3A^+_2 + A_1^+$,

Here also $c=1$ and the same argument as above shows that
\break
$\overline{X'(\R)} =
\KKK
\sqcup \SSS$ and $e(X'(\R)) = 1$.

If $M_0 \ne M_\infty$, the three cusps are transversal. By
Lemma~\ref{lem:monoton}, if
we had a white node, we would have  a black return on $\overline{n}(M_0)$, thus
$e(X'(\R))= 3 -1 +1=3$,
   a contradiction. If we did not have a white node, then $e(X'(\R))
\geq 3 + 0  -1=2$, a
contradiction again.

If $M_0 = M_\infty$, then the node on $\overline{n}(M_1)$ forces a
black return. Since
$M_1$ is a sphere, the contributions to the multiplicity of $\Delta$
impose that the
fibration has at most one real $A_0$ singular fiber and then that all
singular points
are of transversal type. Whence $e(X'(\R)) \geq 3 -1 +1= 3$, a contradiction.

Thus case (\ref{3A2+A1}) is excluded.

\end{itemize}

\section{The Euler number of a real singular Del Pezzo
surface}\label{sec:lasteuler}

Recall that $X'$ is a singular rational elliptic surface obtained
from the Del Pezzo
surface $X$ by blowing up a smooth point; and that a singular point
$A^+_\mu$ is real
analytically equivalent to
$ x^2 + y^2 - z^{\mu + 1} = 0,\ \mu \geq 1$.

Let us denote by $S' \to X'$ the minimal resolution of singularities.

\begin{df} In this paper, we define, for a real variety $X$, $\rho(X)$ to be the
Picard number of the surface $X$ {\em over
$\R$}. It must not be confused with the Picard number of the
complexification $X_\C$ of
$X$. We have always $\rho(X) \leq \rho(X_\C)$ but, generally,
$\rho(X) < \rho(X_\C)$.
\end{df}

\begin{lem}\label{lem:rho} Suppose that the singularities of $X$
(and then of
$X'$) are only of type $A^+_\mu$. Then, denoting by $\# A^+_\mu$
the number of singular points which are of type $A^+_\mu$,
we have:
$$ 2\rho(S') + e\left(S'(\R)\right) - 2\rho(X')  - e\left(X'(\R)\right) =
\sum_{\mu} \mu \cdot (\# A^+_\mu)\;.
$$
\end{lem}

\begin{proof} A local computation shows that
$$
\rho(S') - \rho(X') = \sum_{\mu \textrm{ odd}} \left(1 +\frac {\mu-1}
2\right) +
\sum_{\mu \textrm{ even}} \frac \mu 2
$$  and $e\left(S'(\R)\right) - e\left(X'(\R)\right) = - \# \{ A_\mu,\ \mu
\textrm{ odd}\}$.
\end{proof}

\begin{lem}\label{lem:rho-euler} Let $X$ be a real Du Val Del Pezzo
surface of degree
1. Suppose that $\rho(X) \leq 2$. Suppose moreover that
the singularities are only of type $A^+_\mu$, $\mu \in
\{1,2,3\}$. Then we get for the rational elliptic surface $X'$:
$$ e(X'(\R)) = (8 \quad \textrm{or} \quad 6 ) - \sum_{\mu = 1}^3 \mu
(\# A^+_\mu)\;.
$$
\end{lem}

\begin{proof}

Denote by $\lambda$ the Comessatti characteristic of $S'$ given by
$ 2 \lambda = b_*(S'(\C)) - b_*(S'(\R))
$ (see \cite[Chap. I]{Si89}, and recall that our Betti numbers are taken
with coefficients $\Z/2$).

The nonsingular rational elliptic surface
$S'$ has total Betti number $b_*(S'(\C)) =  e(S'(\C)) = 12$.

Moreover, for a nonsingular
surface $S'$ with
$p_g(S') = 0$ and with $S'(\C)$ simply connected, we have
$ b_1(S'(\R)) = \rho(S') - \lambda
$, see  \cite{Ma97} or \cite{Si89}. Since however $ 2 b_1(S'(\R)) = b_*(S'(\R))
- e (S'(\R))$ we get

$ e\left(S'(\R)\right) + 2  \rho(S') =  b_*(S'(\R)) - 2 b_1(S'(\R)) +
2  \rho(S') =
   b_*(S'(\R)) +   2 \lambda = b_*(S'(\C)) = 12$.

By our hypothesis on the Picard number of the singular Del Pezzo
surface $X$, we have
   $2 \leq \rho(X') \leq 3$, thus the formula follows from Lemma~\ref{lem:rho}.
\end{proof}

\subsection{Exclusion of $2A^+_3 + A_2^+$,  $2A^+_2 + 2A_1^+$ and
$A^+_2 + 3A_1^+$}

In the first case, the branch curve has 2 irreducible components
which are rational.
Indeed one of them is smooth rational and the other has genus 1 and
one singular
point.  Furthermore the two irreducible components intersect in a
real point, thus the
real part of the branch curve is connected.  In the last two cases,
the branch curve
$B$ is irreducible, has genus 4, and has 4 singular points. Thus the
curve is rational
and its real part $B(\R)$ is connected.  It follows that every
connected component $M$
of the topological normalization has the property that
$$ \overline{n}(M) \cap \overline{n}(M_0) \neq \emptyset.$$

Hence $(c+1)$ is the number of connected components of the normalization.

\begin{itemize}
   \item[(\ref{2A3+A2})]
$2A^+_3 + A_2^+$,
$\sum \mu (\# A^+_\mu) = 8$ and $e(X'(\R)) = 0$ or $-2$ by
Lemma~\ref{lem:rho-euler} .

The total multiplicity of $\Delta$ is 12, thus there is at most one
fibre $A_0$ and the
contributions to the Euler characteristic are $1 +1+1+1$, or $1 +
1+1-1$, or $1+1+0$
when the cusp is tangent to a fibre. Thus $e(X'(\R))$ would be
greater than or equal to
2, a contradiction. Thus case (\ref{2A3+A2}) is excluded.
\end{itemize}

\medskip

\begin{itemize}
   \item[(\ref{2A2+2A1})]
$2A^+_2 + 2A_1^+$,
$\sum \mu (\# A^+_\mu) = 6$ and $e(X'(\R)) = 2$ or $0$.

Here, the number of components of $\overline{X'(\R)}$ such that
$\overline{n}(M_j)$  and  $\overline{n}(M_0)$ belong to the same
connected component of
$X'(\R)$  satisfies $1 \leq c \leq 2$.

Assume $c=1$, and do the cylinder smoothing for the nodes, and the
+sphere smoothing
for the cusps. We obtain $b_0 = 3$, and $b_1 \geq 4$, a contradiction.

Assume $c=2$, then there are two cases: $M_0 = M_\infty $ or $M_0 \ne
M_\infty$.  The
topological normalization has $ 3 = c+1$ components, hence
$\overline{X'(\R)} =
\KKK
\sqcup 2\SSS$.

Assume $M_0 \ne M_\infty$, then any cusp is transversal and yields a $(+1)$
contribution to the Euler number. For the component $M_1$, which is
distinct from $M_0$
and from $M_\infty$, we must have a black node, therefore on it there
is also a black
return. In order to get $e(X'(\R)) \leq 2$, there must be a white
return, and then we
should have a white node to make $e(X'(\R)) \leq 2$. But a white
return is necessarily
on $\overline{n}(M_\infty)$, and its existence implies the existence
of other critical
points at $\infty$, a contradiction.

Assume $M_0 = M_\infty$, consider the two components not at $\infty$,
$M_1$ and $M_2$. On them, a white node implies at least two black
returns, while a
black or tangent node implies at least one black return. Since
$\Delta$ has degree 12,
there are exactly two black returns (on each respective $M_j$) and
two black nodes (on
each respective $M_j$). At $\infty$, there are as critical points
only the singular
points, and these are transversal, whence we get 2 transversal cusps,
thus $e(X'(\R)) =
6$, a contradiction.

Thus the configuration $2A^+_2 + 2A_1^+$ does not exist.

\medskip
\item[(\ref{A2+3A1})]
$A^+_2 + 3A_1^+$,
$\sum \mu (\# A^+_\mu) = 5$ and $e(X'(\R)) = 3$ or $1$.

In this case, we have $1 \leq c \leq 3$.

If $c=1$ we have two components $M_0$, $M_1$ and  each node on
$\overline{n}(M_0)$ connects with
$\overline{n}(M_1)$. We do 3 cylinder smoothings and the +sphere
smoothing for the
cusp. We obtain $b_0 = 2$, and $b_1 \geq 6$, a contradiction.

If $c=2$ we have three components $M_0$, $M_1$, $M_2$ and  we let $M_1$ be the
component such that there are two nodes on
$\overline{n}(M_0)$ connecting with
$\overline{n}(M_1)$. We perform 2 cylinder smoothings at these nodes.
For the remaining
node connecting $\overline{n}(M_0)$ with $\overline{n}(M_2)$, we choose the cut
smoothing.  We do the +sphere smoothing for the cusp. This gives $b_0
= 3$, and $b_1
\geq 4$, a contradiction.

Assume $c=3$, and take the normalization. There are 4 components,
whence they are
$\overline{X'(\R)} = \KKK \sqcup 3\SSS$ with $\KKK = M_\infty$.
Notice in particular
that $e(X'(\R)) = 3$.

If $M_0 = M_\infty$, assume that we have a black node. For the corresponding
$\SSS$ component this requires a black return. For a white node, we
need at least two
black returns. For a tangent node, one needs one black return.
However, the number of
$A_0$ fibres is $\leq 3$, thus there is no white node, and we have
exactly 3 black
returns. The contribution to the Euler number is then $e(X'(\R)) \geq
4$. This is a
contradiction which excludes the case $M_0 = M_\infty$.

If $M_0 \ne M_\infty$,  the cusp is transversal hence it  contributes
1 to the Euler
number. Consider the two nodes not involving the component
$M_\infty$. Necessarily they
are black nodes since the other two types of nodes involve the
component $M_\infty$.
These singularities each involve a black return as before. We get a
contribution 5 to
$e(X'(\R)) = 3$. Hence the remaining node and return must contribute
twice a $(-1)$. A
$(-1)$ contribution is white and involves the component $M_\infty$.
But because the
white return gives for the boundary of $\F_{2+}\cap M_\infty$ some
$x$ for which the
degree is 3, and others for which it is 1, there must be another
critical point on
$\F_{2+}\cap M_\infty$, a contradiction.

Thus case (\ref{A2+3A1}) is excluded.
\end{itemize}

\QED {for theorem 1.3}

We end this section with the

\begin{proof}[Proof of Theorem~\ref{teo:main}]
First of all, the reduction  from the case of a geometrically
rational Du Val surface to the case of a Du Val Del Pezzo surface
of degree 1 is done precisely as in \cite{CM08},
Proposition 2.4. and the subsequent proof of Theorem 0.1.
The same argument given in the proof of Cor.0.2 (\cite{CM08},
  end of the third section) shows that  it suffices to consider the 
singular points
of type $A^+_\mu$
which are globally separating when locally separating.  Finally, by
Lemma 1.8 of \cite{CM08}, it remains only to check the case
where $X$ is a
real Du Val Del Pezzo surface of degree 1 with
$\rho(X) \leq 2$.

Then our assertion is exactly reduced to the main assertion of
Theorem~\ref{teo:dps1}.
\end{proof}

\section{Real rationally connected Threefolds}\label{sec:3fold}

This short section explains how Theorem \ref{teo:main} implies  the following.

\begin{lem}\label{lem:question2} Let $ W \to X$ be a real smooth projective
3-fold fibred by rational curves over a geometrically rational
surface $X$. Suppose
that $W(\R)$ is orientable. Then for each connected component $N
\subset W(\R)$, we have
$$
\sum_l (1-\frac 1{n_l(N)})\leq 2\;.
$$
\end{lem}

\begin{proof} Let $ W \to X$ be a real smooth projective
3-fold fibred by rational curves over a geometrically rational
surface $X$. Suppose
that $W(\R)$ is orientable. Let $N
\subset W(\R)$ be a connected component.  Koll\'ar proved 
(see also \cite[3.3, 3.4, and
proof of Cor.~0.2]{CM08}), that there is 

1) a pair of birational contractions
$ c : W \ra W'$, $r :  X \ra X'$, where 

2) $W'$ is a real
projective 3-fold
$W'$ with terminal singularities
such that $K_{W'}$ is Cartier along $W'(\R)$, 

3) $X'$ is a Du Val surface 

4)  a rational curve fibration $f' \colon
W' \to X'$ such that $-K_{W'}$ is $f'$-ample and with

5) $f' \circ c = r \circ f $. 

Let $N''$ be the 
connected component  of the topological normalization
$\overline{W'(\R)}$ such that $N''$ maps onto $ c (\bar{n} (N) ) $.

The main property  of this construction is that 

6) $N'' = N' \#^{a'}\mathbb{P}^3(\mathbb{R})$.

Thanks to \cite[Theorem~8.1]{KoIII}, and \cite[Proof of
Cor.~0.2, end of section 3]{CM08}, 
there is a small perturbation $g \colon N''
\to F$ of $f'\vert_{\overline{n}(N'')}$ such that  $g\vert_{g^{-1}(F\setminus
\partial F)}$ is a Seifert fibration,
and an
injection from  the set of multiple fibres of
$g\vert_{g^{-1}(F\setminus \partial F)}$  to the set of singular
points of $X'$
contained in $f'(\overline{n}(N''))$ which are of type
$A^+$ and globally separating when locally separating. Under this
injection, the
multiplicity of the Seifert fibre equals
$\mu+1$  if  the singular point is of type $A^+_{\mu}$. Hence, the
desired inequality
follows from Theorem \ref{teo:main}.
\end{proof}

\section{Two-dimensional orbifolds}\label{orbi}

In this section we derive first some consequences from our main result
on the components of the topological normalization of a geometrically 
rational
Du Val surface. Then we construct a real smooth algebraic 3-fold whose
real part contains a connected component which is 
 Seifert fibred over the real
projective plane, with two multiple fibres of multiplicity 3.

The first consequence is the following corollary, already mentioned in the
introduction.

\noindent{\bf Corollary \ref{cor:orbi}.} {\em Let $W$ be a real smooth
projective rationally connected
3-fold fibred by rational curves. Suppose that $W(\R)$ is orientable
   and let $N$ be a connected
component of $W(\R)$. Then neither $N$ nor $N'$ can be endowed with a
$\widetilde{\operatorname{SL}_2\R}$ structure or with a $\HH^2\times
\R$ structure
whose base orbifold $F$  is orientable.}

\begin{proof}
As we already mentioned,

1)  if a 3-manifold  possesses a geometric
structure, then the corresponding geometry
  is unique,

2) every Seifert fibred manifold admits a geometric structure.

Moreover,

3) if $N$ or $N'$ can be endowed with a
$\widetilde{\operatorname{SL}_2\R}$ structure or with a $\HH^2\times
\R$ structure, then $N$ is Seifert fibred and, by the cited theorem of Milnor,
we have that $N'$ is Seifert fibred by the given rational curve fibration.

Now, the six geometries for  Seifert
fibrations are distinguished by negativity, nullity or positivity of the Euler
characteristic $\chi_{top}(F)$ of the base orbifold
   and by the vanishing or non vanishing  of the Euler
number  of the Seifert bundle~\cite[Table~4.1]{Scott}. In
particular  the $\widetilde{\operatorname{SL}_2\R}$ and the $\HH^2\times
\R$ geometry correspond exactly to the `hyperbolic' case,
where $\chi_{top}(F)$ is negative.

We conclude then by virtue  of Theorem \ref{teo:multiple}.
\end{proof}

\begin{prop}\label{proporbi} Let $N$ be as in  Corollary \ref{cor:orbi}.
Suppose moreover
that $N$ admits a Seifert fibration with base orbifold $F$ such that
$\vert F\vert$ is
orientable. Then the Euler characteristic $\chi_{top}(F)$ of the compact
2-dimensional orbifold $F$ is
nonnegative.
\end{prop}

\begin{proof} By \cite[Theorem~4.3 and Lemma~4.4]{CM08}, the
topological normalization
$\overline{X(\R)}$ can be realized as the real part of a real perturbation
   $X_\varepsilon$
of
$X$. Thanks to Comessatti's Theorem, an orientable connected component
of $X_\varepsilon(\R)$ is a
sphere or a torus. In the last case, the Seifert fibration $N\to F$
has no singular fibre and $F$
is a manifold, hence the Euler characteristic of $F$ is zero. In the
latter case, the Euler
characteristic of $F$ is positive.
\end{proof}

\subsection{A Seifert fibration with base orbifold of hyperbolic
type.}

As announced in the introduction, we are going to construct
  a real smooth 3-fold $W$,  fibred by
rational curves over a  Du Val Del Pezzo surface $X$ , with the property
that  $W(\R)$ is connected and enjoys the following properties:

i)  $W(\R)$ is orientable,

ii) $W(\R)$ has a connected component which is Seifert fibred over a
base orbifold  $F$,

iii) $F$ is non orientable and of hyperbolic type

iv) the Du Val Del Pezzo surface $X$ is minimal over $\R$.

Our method of construction is based on a rather general procedure
which produces Seifert fibrations as projectivized tangent bundles
of Du Val surfaces, so we
start with some easy lemmas, the first one being well known.

\begin{lem}\label{orient}
Let $M$ be a real differentiable manifold. Then the tangent $ TM$ is always
orientable, while $\PP (TM)$ is orientable if $ n : = dim_{\R} (M)$ is even.
\end{lem}

\begin{proof}
Let $ p : TM \ra M$ be the natural projection.

By the  exact sequence
$ 0 \ra p^* (TM) \ra T (TM) \ra  p^* (TM) \ra 0$ we get that
$\bigwedge^{2n} (T(TM)) \cong  \bigwedge^{n} (p^*(TM))^{\otimes 2}$
is trivial.

Let $ \pi : \PP (TM) \ra M$ be the natural projection.

Then by
the  exact sequences

$ 0 \ra VT (\PP (TM)) \ra T (\PP (TM)) \ra  \pi^* (TM) \ra 0$,

(here $VT$ denotes the subbundle of vertical vectors) and

$ 0   \ra ( \R \times  \PP (TM)) \ra  \pi^* (TM) \otimes U^{-1} \ra VT (\PP
(TM))\ra 0$,

where $U$ is the tautological line subbundle,
we get  $\bigwedge^{2n-1} T (\PP (TM)) \cong  \bigwedge^{n}
(\pi^*(TM))^{\otimes 2} \otimes U^{ \otimes n}$,
thus we have a trivial line bundle if $n$ is even.
\end{proof}

Next, we consider the projectivized tangent bundle of Du Val surfaces with
$A_n$ singularities.

This 3-fold is simply obtained by glueing together the projectivized
tangent bundle of the smooth part with the $\mu_{n+1}$ quotient
$$ Y_n : = (\bbA_{\C}^2 \times \PP^1_{\C} ) / \mu_{n+1}$$ of
the projectivized tangent bundle of  the affine plane via the action of
the $(n+1)$-th roots of unity induced by the action on $\bbA_{\C}^2$
yielding
the quotient $ A_n : = \bbA_{\C}^2  / \mu_{n+1}$.

\begin{lem}
$Y_n$ has isolated singularities if and only if $n$ is even.
If $n$ is even, these singularities are terminal quotient singularities
$Z_n : = \frac{1}{n+1}(1,-1,2)$ where the canonical divisor is not Cartier.

\end{lem}

\begin{proof}
$\mu_{n+1} : = \{ \zeta | \zeta^{n+1} = 1 \}$ acts on the affine plane
$\bbA_{\C}^2$ by $(x,y) \mapsto (\zeta x ,  \zeta^{-1} y)$, whence
its action on $\bbA_{\C}^2 \times \PP^1_{\C}$,
$$(x,y)(\xi : \eta ) \mapsto (\zeta x ,  \zeta^{-1} y)(\zeta \xi :  \zeta^{-1}
\eta ).$$

If $n$ is odd, $ n+1 = 2 k $ and $\zeta^k$ acts trivially on $\PP^1_{\C}$;
we see that we get a 
corresponding 1-dimensional singular locus,  analytically
isomorphic to $A_1 \times \bbA_{\C}^1$.

Assume now that $n$ is even, so that each nontrivial group element has
only two fixed points, namely, for $ x=y=\xi=0$, respectively for
$ x=y=\eta=0$. At each point, passing to local coordinates,
we see that we have a singularity of type $Z_n$, the quotient 
$ Z_n : = \bbA_{\C}^3  / \mu_{n+1}$
by the action where
$(x,y,z) \mapsto (\zeta x ,  \zeta^{-1} y,  \zeta^{2} z)$.
This singularity is well known to be terminal (see  \cite{mori}), and the
Zariski canonical divisor $K_Z$ there
is not Cartier because the differential form $dx \wedge dy \wedge dz $
is not invariant, being multiplied by $\zeta^{2}$ (only $(n+1) K_Z$ 
is Cartier).

\end{proof}

Over the real numbers, however, we have different  forms of the
$A_n$ singularities, as we mentioned in the beginning.

\begin{lem}\label{PTB}
Let $n$ be an even number and define $Y_n^-$ to be the projectivized
tangent bundle of a singularity of type $A_n^-$,
and define analogously $Y_n^+$ for a  singularity of type $A_n^+$.
$Y_n^{\pm }$ has terminal isolated singularities and the  real part 
$Y_n^-(\R)$ is a PL-manifold  of real dimension 3, while the  real part 
$Y_n^+(\R)$ is contained in the smooth locus of
  $Y_n^+$.

The natural projection $Y_n^{+ }(\R) \ra A_n^{+ }(\R )$
is a Seifert fibration with a multiple fibre of multiplicity
$(n+1)$ over the origin, while  $Y_n^{- }(\R) \ra A_n^{- }(\R )$
is a  topologically trivial $S^1$-bundle. 
\end{lem}

\begin{proof}

We treat first the $A_n^-$-case.
We consider the real group scheme $\mu_{n+1}^- : = \{ \zeta | \zeta^{n+1} =
1\}$ which acts on the affine plane
$\bbA_{\R}^2$ by $(x,y) \mapsto (\zeta x ,  \zeta^{-1} y)$, whence
its action on $\bbA_{\R}^2 \times \PP^1_{\R}$,
$$(x,y)(\xi : \eta ) \mapsto (\zeta x ,  \zeta^{-1} y)(\zeta \xi :  \zeta^{-1}
\eta )$$
is such that each nontrivial group element has
only two fixed points, namely, the point where $ x=y=\xi=0$, 
respectively the one where
$ x=y=\eta=0$. At each point, passing to local coordinates,
we see that we have a singularity of type $Z_n^-$, the quotient 
$ Z_n^- : = \bbA_{\R}^3  / \mu_{n+1}^-$
by the action where
$(x,y,z) \mapsto (\zeta x ,  \zeta^{-1} y,  \zeta^{2} z)$.

Let us now observe that  $Z_n^-$ sits inside a Galois sandwich
$$ \bbA_{\R}^3  \stackrel{\psi_2}{\ra}   Z_n^- \stackrel{\psi_1}{\ra} 
\bbA_{\R}^3  $$
where $\psi_2$ is the quotient morphism and the composition
$\Phi : = \psi_1 \circ \psi_2$ is given by 
$$ \Phi (x,y,z) : = (x^{n+1},y^{n+1},z^{n+1})  $$
(the coordinates of $\psi_2$ are just a set of invariant monomials
including $x^{n+1},y^{n+1},z^{n+1}, xy , y^2 z $).
Since $\Phi$ induces a homeomorphism $\Phi (\R) \colon \bbA_{\R}^3
\ra \bbA_{\R}^3$, our claim is established if we show that in the real
part of the sandwich
$$ \R^3  \stackrel{\psi_2 (\R)}{\ra}   Z_n^-(\R) \stackrel{\psi_1(\R)}{\ra} 
\R^3  $$
the polynomial map $\psi_2 (\R)$ is surjective.

Take a point $P \in Z_n^-(\R)$: since it maps under $\psi_1(\R)$ to 
$\R^3  $, there exists a real point $(x,y,z) \in \R^3  $ and
elements $\zeta_i \in \mu_{n+1}$, for $ i=1,2,3$, such that 
$P = \psi_2(\zeta_1 x, \zeta_2 y , \zeta_3 z) $.
Since however $\zeta_1 x \zeta_2 y \in \R$ and $(\zeta_2 y )^2 \zeta_3 z
\in \R$, we get: $\zeta_1  \zeta_2  \in \R$, $(\zeta_2  )^2 \zeta_3 
\in \R$. Since $n+1$ is odd, then $\zeta_2 = \zeta_1  ^{-1}$
and $\zeta_3 = \zeta_2  ^{-2} =  \zeta_1  ^{2}$: we have thus proven that
$P = \psi_2( x,  y ,  z) $.

Similarly, we see that the quotient morphism $\R^2 \ra A_n^- (\R)$
is a homeomorphism. Hence, the product fibration 
$\R^2 \times \PP^1_{\R} (\R)$ descends to a topologically trivial
$S^1$-bundle  over $A_n^- (\R)$. 

The case of   the $A_n^+$-case is simpler but more interesting. The action of
$\mu_{n+1}(\C) : = \{ \zeta \in \C | \zeta^{n+1} = 1 \}$  on the affine plane
$\bbA_{\C}^2$ is given by
 $$(x + iy, x-iy) \mapsto (\zeta (x + iy),  \zeta^{-1} (x-iy) ).$$

The action is defined over $\R$ since
$$(x , y) \mapsto ( Re (\zeta )x  -  Im (\zeta )y, Im (\zeta )x + Re (\zeta )y ),$$
and it defines an action of the real group scheme 
$\mu_{n+1}^+ : = \{ (a,b) | (a + ib)^{n+1} = 1 \}$ on 
$\bbA_{\R}^2$  given by
 $$(x , y) \mapsto (a x  -  b y, b x +  a y).$$

The ring of real invariant polynomials is generated, if we set 
$ P : = (x + i y )^{n+1} $, by $ u : = (P + \bar{P}), v := \frac{1}{i}(P - \bar{P}), 
w : = (x^2 + y^2)$, which satisfy the equation of $A_n^+$,
$  u^2 + v^2 = w^{n+1}$.

 The cyclic group stabilizes $\R^2$, and the origin is the only fixed point,
while the action on $\bbA_{\R}^2 \times \PP^1_{\R}$
has no real fixed points, hence $Y_n^{+ }(\R) \ra A_n^{+ }(\R )$
is a Seifert bundle and the multiplicity over the origin is $n+1$. 
\end{proof}

\begin{oss}
As a consequence of the previous lemma, given any real Du Val surface
$X$ with only $A_n^+$ singularities with $n$ even, the projectivized
tangent bundle of $X$, $ W' : = \PP (TX)$ is a 3-fold with terminal
singularities,  such that 

i) the real part $ W' (\R)$ is contained 
in the smooth locus of $W'$,

ii)  $ W' (\R)$ is  orientable.
\end{oss}

The previous remark allows us to  construct the desired
 real 3-fold.

\noindent{\bf Theorem \ref{hyperbolic}.}
{\em There exists a minimal real Du Val Del Pezzo surface
$X$ of degree $1$ having exactly two singular points, of type $A_2^+$,
and such that the real part $X(\R)$ has a
connected component containing the two singular points and which is
homeomorphic to a  real projective plane.

Let $W'$ be the projectivized tangent bundle of $X$: then
$W'$ has terminal singularities, $ W' (\R)$ is contained 
in the smooth locus of $W'$, 
in particular if $W$ is obtained  resolving  the singular points of $W'$,
then $ W (\R) =  W' (\R)$.

Moreover $ W (\R)$ is  orientable and  contains a connected component $N$
which is  Seifert fibred over a non
orientable orbifold of hyperbolic type (the real projective plane with two
points of multiplicity 3). } 

\begin{proof}
We construct the real Del Pezzo surface $X$ as the blow up of the real
projective plane in $8$ real points. We shall indeed construct a family
of such surfaces $Y$, having two real $A_2^-$ singularities, and two 
real and non isolated $A_1$
singularities. For certain values of the parameters, once we represent the  
Del Pezzo surface $Y$ as the double cover of the quadric cone $Q$ branched
on the vertex of the cone and on a real branch curve $B$, then the two
$A_1$ points  give rise to two isolated real points of the real part $B(\R)$ of the branch curve.

Using then the generalization of Brusotti's theorem given in Theorem 4.3 of
\cite{CM08}, we can take a small deformation which leaves unchanged 
the two real
$A_2^-$ singularities, but deforms  the two
$A_1$ points replacing the two isolated  points of $B(\R)$ by two small ovals.

We obtain now a real Del Pezzo surface $Z$ of degree 1 with exactly two
real $A_2^-$ singularities, and our desired real Del Pezzo surface $X$ 
will be the same complex surface $Z$, but with a new real structure 
$\sigma' : = \sigma \circ i$, where $\sigma$ is the real involution of $Z$,
and $i : Z \ra Z$ is the {\bf Bertini involution},  the covering involution of
the bianticanonical morphism, yielding $Z$ as a double cover of the
quadric cone $Q$.

In terms of this last representation, this simply amounts to exchanging the
region of positivity with the region of negativity. For this reason, $X$
has now two $A_2^+$
singularities, and the inside of the two ovals are now regions of positivity;
we conclude that the connected components of $ X (\R)$ consist of
two spheres $S^2$, and of a component homeomorphic 
to the real projective plane, and containing the two $A_2^+$
singularities.

In order to provide the defining equations for these Del Pezzo surfaces,
consider now  the plane
$\R^2$ with coordinates
$(u,v)$ and in it the two pairs of lines 

$$  L_+ \cup L_- : = \{ u^2 - v^2 = 0\} $$
$$  L_1 \cup L_{-1} : = \{ u^2 - 1 = 0\}. $$

Each pair of lines shall yield a respective $A_2$ configuration on our Del
Pezzo surfaces of degree $1$. We choose in fact eight points in the plane,
 such
that each of the four lines contains four of them, namely the set 
$$ \{ (1,1), (1,-1),(-1,1),(-1,-1), (1+ \de,1+\de), (-1- \de,1+\de),(1,1
+\epsilon),(-1,1
+\epsilon)
\}.
$$
We see easily that the configuration is symmetric with respect to the 
real involution $ i $ such that $ i (u,v) = (-u,v)$, and this symmetry
is responsible of the fact that there is a conic $D$ containing the 
following symmetrical set of six points:
$ \{ (1,1), (-1,1), (1+ \de,1+\de), (-1- \de,1+\de),(1,1
+\epsilon),(-1,1
+\epsilon)
\}.$

Similarly, there  is a conic $D'$ containing the 
following symmetrical set of six points:
$ \{ (1,-1), (-1,-1), (1+ \de,1+\de), (-1- \de,1+\de),(1,1
+\epsilon),(-1,1
+\epsilon)
\}.$

We let $\tilde{Y}$ be the blow up of the real projective plane in the above
eight points. We claim that $\tilde{Y}$ is a weak Del Pezzo surface of
degree $1$, i.e., that its anticanonical divisor is nef (it is big 
since it clearly satisfies
$K^2_{\tilde{Y}} = 1$). 

This claim follows right away from the fact that, if we take homogeneous
coordinates $(u,v,t)$ on $\PP^2$, the system of cubics through the eight 
points is the pencil
spanned by the two cubics
$$ x_0 : = (u^2 - v^2) (v- (1+\epsilon)t),  \ x_1 :=  (u^2 - t^2) (v-(1+\de)t),$$
whose proper transforms meet only (transversally) in the point `at infinity'
$v=t=0$.

The bianticanonical morphism $\phi$ of $\tilde{Y}$ is the double covering
of the quadric cone $Q = \PP(1,1,2) $ given by $(x_0, x_1, y_2)$, where
$$y_2 : = (u^2 - v^2) (u^2 - t^2) ((\de v - \epsilon (1 + \de)t)^2 - (\de -
\epsilon )^2 u^2 ) $$
has as set of zeros the union of six lines and passes doubly through the
$8$ points, but does not vanish on the base point of the anticanonical pencil
(hence, $y_2$ is not a linear combination of $x_0^2, x_0 x_1, x_1^2$).

The morphism  $\phi$ clearly factors through the quotient of $\tilde{Y}$ by the
involution $i$ ($\phi (u,z,t) = \psi (u^2, v,t)$) and we have then a
factorization $ \phi = \psi \circ \pi$, where 
$\psi :\tilde{Y} / i \ra Q$, and $\pi$ is the quotient projection.
Hence it follows that
$\psi$ contracts the images under $\pi$ of the $(-2)$-curves on $\tilde{Y}$,
and that the branch curve $ B \subset Q$ is the image under $\phi$
of the projective line $ u = 0$.

The branch curve $B$ is irreducible of arithmetic genus $4$, and it has $4$
singular points, corresponding to the blow down of the curves $ D, D', L_+
\cup L_-, L_1 \cup  L_{-1}$: hence we conclude that the only $(-2)$-curves
on $\tilde{Y}$ lie on the corresponding fibres of the anticanonical pencil, and
a direct inspection shows that there are no other $(-2)$-curves on $\tilde{Y}$.

 We want now to find a choice of the parameters such that  the curves $D, D'$
are real and do not intersect the line $u=0$, thereby yielding two real
isolated double points of the branch curve $B(\R)$.

These conditions lead to some inequalities holding among $ \de, \epsilon$,
but a simple solution is obtained choosing
$$ \de = 1, \epsilon = -1 . $$

For this choice we get 

$$ x_0  = (u^2 - v^2) v,  \ x_1 =  (u^2 - t^2) (v- 2t), y_2 = (u^2 - v^2)
(u^2 - t^2) ((2t+v)^2 - 4 u^2) $$
$$ D = \{ -2 (u^2 - t^2) + 3 v (v-t) = 0\},  D' = \{ -2 (u^2 - t^2) +  v (v+t) =
0.$$ 
Whence, the line $u=0$ meets $D$ and $D'$ in two pairs of complex
conjugate points.

\begin{figure}[htbp]
\begin{center}
     \epsfysize=4.5cm
    \epsfbox{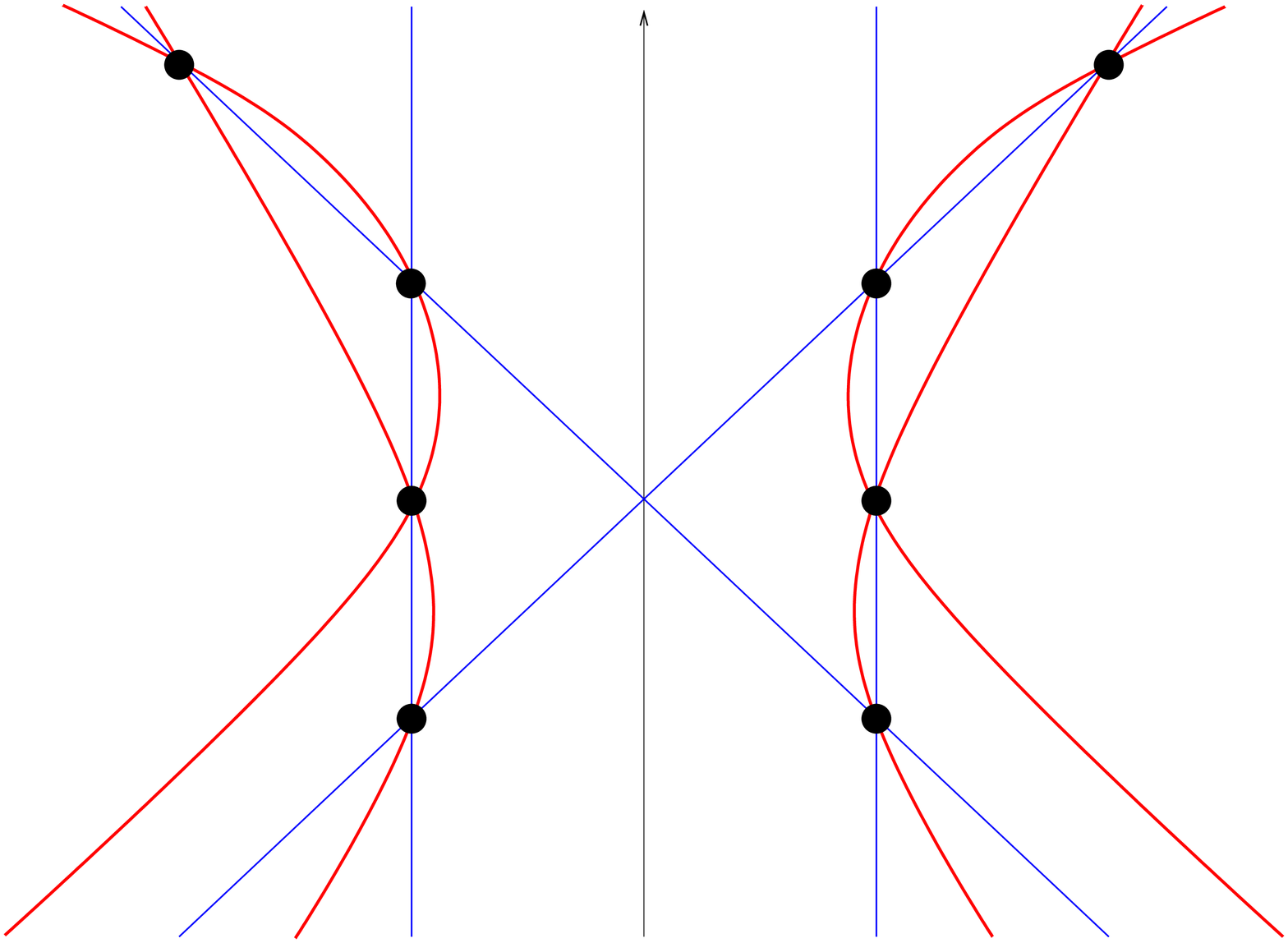}
\caption{}\label{fig:branch-curve}
\end{center}
\end{figure}

 Using now the plane representation of the quadric cone (a main
tool in
\cite{CM08}) given by the rational map $ (x_0^2, x_0 x_1, y_2)$,
we see that the line $ u=0$ maps to
$ x  :=  x_1 / x_0  , y : = y_2 / x_0^2$, and using the inhomogeneous
coordinate
$v$ on the line $u=0$,  we obtain $ x = \frac{v-2}{v^3}$, $ y =
\frac{(v+2)^2}{v^4}
\geq 0$.

We get $y \geq 0$, $ y=0$ for $ v = \infty, v = -2$, and the corresponding
points of $B$ are: $(0,0)(1/2, 0 )$. 

Choosing $t$ as local coordinate at $ v = \infty$, we get $ x = t^2 (1-2t), y =
t^2 (2t + 1)^2$, whence at $(0,0)$ $B$ has an ordinary  cusp 
with non vertical
tangent.

The rational function $x$ only ramifies for $v= 0, 3, \infty$,
and for $ v=3$ we get the point $ x = \frac{1}{27}, y = \frac{25}{81} $.
For $v=0$ we must go to the other chart of our Segre-Hirzebruch surface
$\F_2$, and setting $x' = x_0 / x_1, y' = y_2 / x_1^2$, we get
the parametrization $ x' = \frac{v^3}{v-2}, y' = \frac{v^2 (v+2)^2}{(v-2)^2}$,
showing that the point $x' = y' = 0$ of $B$ is an ordinary   cusp with vertical
tangent.

Observe however that, since  on our surface $\tilde{Y}$ 
$$y = v^{-2}(u^2 - v^2)^{-1}
(u^2 - 1) ((2+v)^2 - 4 u^2), $$
and as $ u \ra \infty$ $ y \ra - \infty$, the image of  $\tilde{Y}(\R)$ 
(region of positivity for the double cover) is the one beneath the curve $B(\R)$.
Moreover, the isolated real double points of the branch curve $B(\R)$ are not
isolated points of the real Del Pezzo surface  $\tilde{Y}(\R)$, whence they lie in the
region of positivity.
Since we blew up $8$ real points in the plane, we have for the Euler
characteristic  (which we calculate through the $\Z / 2$ Betti numbers):
$e(\tilde{Y} (\R)) = 1-8= -7.$

After contracting the six $(-2)$-curves to the two $A_2^-$ singularities and the 
two
$A_1$-singularities, we obtain a Du Val Del Pezzo surface $Y$ 
 with $e(Y(\R)) = - 1.$

As already announced, we obtain a Du Val Del Pezzo surface $Z$
with two $A_2^-$ singularities using a small deformation on $Y$ which
realizes  the following local deformations: no local deformation at the 
$A_2^-$ singularities, and, at the
$A_1$ singularities, which are of the form $z^2 = a^2 + b^2$, for suitable
local coordinates $(a,b)$ on the quadric cone $Q$, a small deformation
of type  $z^2 = a^2 + b^2 - r^2$, where $r$ is a small real number.
The existence of this global deformation is  guaranteed by  \cite[Theorem
4.3]{CM08}; in practice, some simple calculations show that this
deformation is
obtained by simply moving the $8$ points on the four lines 
 $  L_+ , 
 L_-, L_1 ,  L_{-1}$ but breaking the symmetry $ u \mapsto -u$.

 This deformation produces on the one side a
real Du Val  Del Pezzo
$Z$ with
$e(Z(\R)) = - 3,$ on the other side it produces two real ovals 
in the branch curve
$B(\R)$, whose respective interiors are now in the region of negativity.

\begin{figure}[htbp]
    \centering
   \epsfysize=2.2cm
    \epsfbox{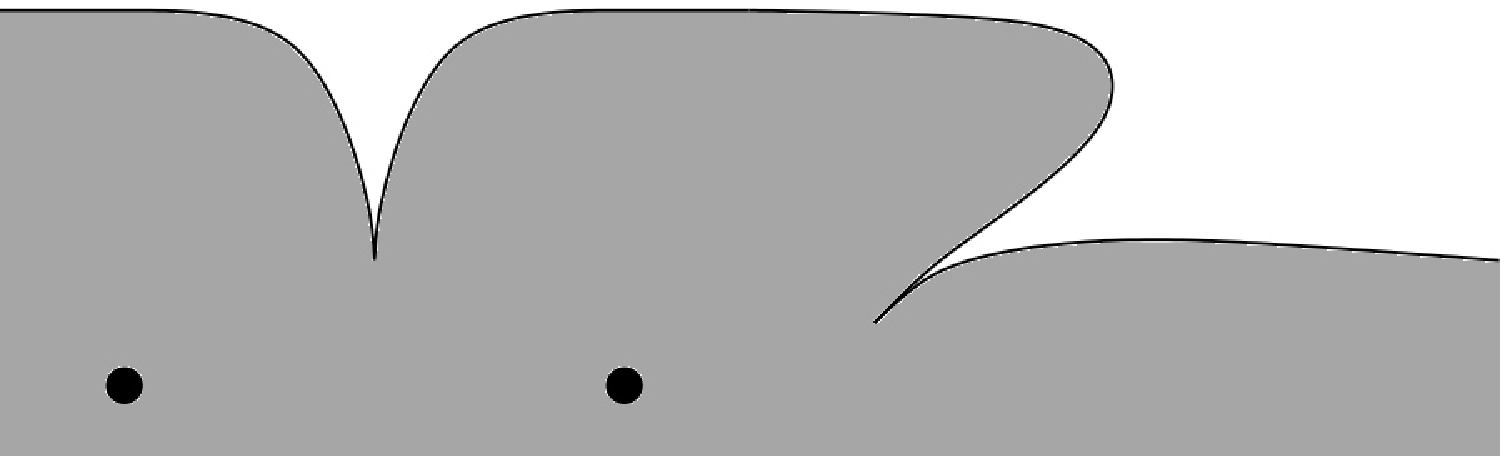}
\quad
 \epsfysize=2.2cm
    \epsfbox{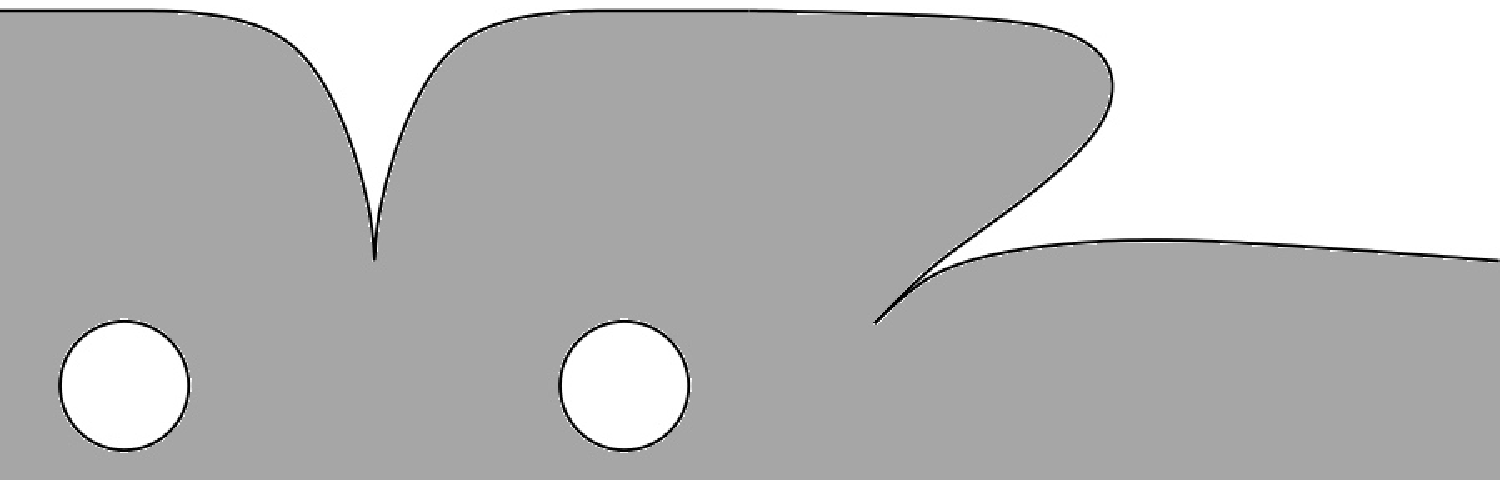}
\caption{}
    \label{fig:2cusps}
    
  \end{figure}

We define now $X$ as the same complex surface as $Z$, 
but with real structure
$\sigma \circ i$, $i$ being the Bertini involution. Now the region of
positivity and negativity are exchanged, since
 the old real function $z$ is replaced by $iz: = w$; hence, instead of having
$z^2 = f(x,y)$, we obtain $w^2 = - f(x,y)$. 

Clearly $X(\R)$ consists of a real projective plane with two $A_2^+$
singularities, together with two spheres.
The ($\R$-) minimality of $X$ (see e.g. \cite{ko-surf}) is a consequence of the following 

\begin{lem}
$X$ has real Picard number $\rho (X) = 1$.
\end{lem}

\Proof
Let $S$ be the minimal resolution of singularities of $X$,
and observe that $\rho (S) = \rho (X)  + 2$, since the blow up
of a real singular point of type $A_2^+$ yields a pair of complex
conjugate $(-2)$-curves.
Arguing as in Lemma \ref{lem:rho-euler} we calculate 
$$ \rho (S) =b_1 (S (\R))  + \lambda (S)= 1 + \lambda (S).$$
Here, by the definition of $\lambda$, 
$$ 2 \lambda (S) = b_*(S (\C)) -  b_*(S (\R)) = 11 - (2+2+3) = 4.$$

Whence $  \lambda (S)=2$, $ \rho (S) = 3$, $\rho (X) = 1$.
\QED for the Lemma.

The remaining assertions follow from lemma \ref{PTB}, 
and most of them were
already mentioned in the previous remark: 
  observe finally that a real projective plane with two points
of multiplicity $3$ is a orbifold of hyperbolic type, 
since $ 1 - \frac{2}{3} - \frac{2}{3} = - \frac{1}{3} < 0.$

\end{proof}

\vfill

\noindent {\bf Author's address:}

\bigskip

\noindent  Fabrizio Catanese\\ Lehrstuhl Mathematik VIII\\
Universit\"at Bayreuth,
NWII\\ D-95440 Bayreuth, Germany

e-mail: Fabrizio.Catanese@uni-bayreuth.de

\noindent Fr\'ed\'eric Mangolte\\ Laboratoire de Math\'ematiques\\
Universit\'e de
Savoie\\ 73376 Le Bourget du Lac Cedex, France

e-mail: Frederic.Mangolte@univ-savoie.fr

\end{document}